\documentclass[12pt]{amsart}
\usepackage{amsmath}
\usepackage{amsthm}
\usepackage{amsfonts}
\usepackage{amssymb}
\usepackage{latexsym}
\usepackage{yfonts}
\usepackage{enumerate}
\usepackage{graphicx}
\usepackage{setspace}

\newcommand{\Z}{\mbox{$\mathbb Z$}}

\newcommand{\R}{\mbox{$\mathbb R$}}     

\newtheorem{theo}{Theorem}[section]

\newtheorem{lem}[theo]{Lemma}
\newtheorem{prop}[theo]{Proposition}

\newtheorem{conj}[theo]{Conjecture}
\newtheorem{cor}[theo]{Corollary}

\begin{document}

\title[Stolarksy's conjecture]{Stolarsky's conjecture and the sum of digits of polynomial values}

\author{Kevin G. Hare}
\address{Department of Pure Mathematics,  University of Waterloo, Waterloo, Ontario,  Canada,  N2L 3G1,}
\email{kghare@math.uwaterloo.ca}
\thanks{K.G. Hare was partially supported by NSERC}
\thanks{Computational support provided by CFI/OIT grant}

\author{Shanta Laishram}
\address{Department of  Mathematics,
Indian Institute of Science Education and Research, Bhopal, 462 023, India,}
\email{shanta@iiserbhopal.ac.in}

\author{Thomas Stoll}
\address{Institut de Math\'ematiques de Luminy, Universit\'e de la M\'editerran\'ee, 13288 Marseille Cedex 9, France,}
\email{stoll@iml.univ-mrs.fr}
\thanks{Th. Stoll was partially supported by an APART grant of the Austrian Academy of Sciences}


\maketitle

\begin{abstract}
  Let $s_q(n)$ denote the sum of the digits in the $q$-ary expansion of an integer $n$.
  In 1978, Stolarsky showed that
  $$ \liminf_{n\to\infty} \frac{s_2(n^2)}{s_2(n)} = 0. $$ He conjectured that,
  as for $n^2$, this limit infimum should be 0 for higher powers of $n$. We prove and
  generalize this conjecture showing that for any polynomial
  $p(x)=a_h x^h+a_{h-1} x^{h-1} + \dots + a_0 \in \Z[x]$ with $h\geq 2$ and $a_h>0$
  and any base $q$, \[ \liminf_{n\to\infty} \frac{s_q(p(n))}{s_q(n)}=0.\]
  For any $\varepsilon > 0$ we give a bound on the minimal $n$ such that the ratio
  $s_q(p(n))/s_q(n) < \varepsilon$. Further, we give lower bounds for the number of
  $n < N$ such that $s_q(p(n))/s_q(n) < \varepsilon$.
\end{abstract}

\section{Introduction}

Let $q\geq 2$ and denote by $s_q(n)$ the sum of digits in the $q$-ary representation of an integer $n$. In recent years, much effort has been made to get a better understanding of the distribution properties of $s_q$ regarding certain subsequences of the positive integers. We mention the ground-breaking work by C.~Mauduit and J.~Rivat on the distribution of $s_q$ of primes~\cite{MR09-1} and of squares~\cite{MR09-2}. In the case of general polynomials $p(n)$ of degree $h\geq 2$ very little is known. For the current state of knowledge, we refer to the work of C.~Dartyge and G.~Tenenbaum~\cite{DT06}, who provided some density estimates for the evaluation of $s_q(p(n))$ in arithmetic progressions.  The authors \cite{HLS} recently examined the special case when $s_q(p(n)) \approx s_q(n)$.

A problem of a more elementary (though, non-trivial) nature is to study extremal properties of $s_q(p(n))$. Here we will always assume that
\begin{equation}\label{genpoly}
  p(x)=a_h x^h+a_{h-1} x^{h-1} + \dots + a_0 \in \Z[x]
\end{equation}
is a polynomial of degree $h\geq 2$ with leading coefficient $a_h>0$.

In the binary case when $q = 2$, B.~Lindstr\"om~\cite{Li97} showed that
\begin{equation}\label{lind}
  \limsup_{n \to \infty} \frac{s_2(p(n))}{\log_2 n}= h.
\end{equation}
In the proof of~(\ref{lind}), Lindstr\"om uses a sequence of integers $n$ with many $1$'s in their binary expansions such that $p(n)$ also has many $1$'s.
The special case $p(n)=n^2$ of~(\ref{lind}) has been reproved by M.~Drmota and J.~Rivat~\cite{DR05} with constructions due to J. Cassaigne and G. Baron.

On the other hand, it is an intriguing question whether it is possible to generate infinitely many integers $n$ such that $p(n)$ has few $1$'s compared to $n$. If this is possible, then this is indeed a rare event.
It is well-known~\cite{De75, Pe02} that the average order of magnitude of $s_q(n)$ and $s_q(n^h)$ is
\begin{equation}\label{avorder}
   \sum_{n<N} s_q(n)\sim \frac{1}{h}\sum_{n<N} s_q(n^h)\sim \frac{q-1}{2\log q} \; N\log N.
\end{equation}
In particular, the average value of $s_q(n^h)$ is $h$ times larger than the
    average value of $s_q(n)$.

In 1978, K. Stolarsky~\cite{St78} proved several results on the extremal values of $s_q(p(n))/s_q(n)$ for the special case when $q = 2$ and $p(n) = n^h$. He showed that the maximal order of magnitude is $$c(h) (\log_2 n)^{1-1/h},$$ where $c(h)$ only depends on $h$. This result is best possible, which follows from the Bose-Chowla theorem~\cite{BC62, HR83}. His proof can be generalized to base $q$ and to general polynomials $p(n)$.
Although this generalization is straightforward, we include it here for completeness. Recall that $p(n)$ may have negative coefficients as well.
\begin{theo}\label{limsupsto} Let $p(x) \in \Z[x]$ have degree at least $2$ and positive leading coefficient.
  \begin{enumerate}
  \item If $p(n)$ has only nonnegative coefficients then there exists $c_1$, dependent only on $p(x)$ and $q$, such that for all $n\geq 2$,
  $$\frac{s_q(p(n))}{s_q(n)}\leq c_1 (\log_q n)^{1-1/h}.$$
  This is best possible in that there is a constant $c_1'$, dependent only on $p(x)$, such that
  $$\frac{s_q(p(n))}{s_q(n)}> c_1' (\log_q n)^{1-1/h}$$
  infinitely often.
  \item If $p(n)$ has at least one negative coefficient then there exists $c_2$ and $n_0$, dependent only on $p(x)$ and $q$,
   such that for all $n\geq n_0$,
  $$\frac{s_q(p(n))}{s_q(n)}\leq c_2 \log_q n.$$
  This is best possible in that for all $\varepsilon>0$ we have
  $$\frac{s_q(p(n))}{s_q(n)}> (q-1-\varepsilon) \log_q n$$
  infinitely often.
  \end{enumerate}
\end{theo}
The proof of this result along with some useful preliminary results are
    given in Section \ref{sec:prelim}.

For the minimal order of $s_q(p(n))/s_q(n)$, Stolarsky treated the special case $q=2$ and $p(n)=n^2$. He
proved that there are infinitely many integers $n$ such that
\begin{equation}\label{stolthm}
  \frac{s_2(n^2)}{s_2(n)}\leq \frac{4 (\log \log n)^2}{\log n}.
\end{equation}
He conjectured that an analogous result is true for every fixed $h\geq 2$ but he did ``not see how to prove this''.
\begin{conj}[Stolarsky~\cite{St78}, 1978]\label{conj}
  For fixed $h\geq 2$,
  $$ \liminf_{n\to \infty} \frac{s_2(n^h)}{s_2(n)}=0.$$
\end{conj}
By naive methods, it can be quite hard to find even a single value
    $n$ such that $s_2(n^h)<s_2(n)$ for some $h$, let alone observe that the
    limit infimum goes to $0$.
For example, an extremely brute force calculation shows that
    the minimal $n$ such that $s_2(n^3) < s_2(n)$ is
    $n=407182835067\approx 2^{39}$.

In Section \ref{sec:mtheo} we  prove and generalize
    Conjecture~\ref{conj}.
\begin{theo}\label{mtheo}
  We have
  $$ \liminf_{n\to \infty} \frac{s_q(p(n))}{s_q(n)}=0.$$
\end{theo}

In view of our generalization, it is natural to ask how quickly we can expect this ratio to go to zero. Recall that $h = \deg p$.
\begin{theo}\label{mtheo2}
There exist explicitly computable constants $B$ and $C$,
    dependent only on $p(x)$ and $q$,
    such that for all $\varepsilon$ with $0<\varepsilon<h(4h+1)$
    there exists an $n < B \cdot C^{1/\varepsilon}$ with
    \[ \frac{s_q(p(n))}{s_q(n)} < \varepsilon. \]
\end{theo}
The proof of this result along with an explicit construction for $B$ and $C$
    is given in Section \ref{sec:mtheo2}.
As a nice Corollary to this result we have
\begin{cor}
There exists a constant $C_0$, dependent only on $p(x)$ and $q$,
    such that there exists infinitely many $n$ with
    \[ \frac{s_q(p(n))}{s_q(n)} \leq \frac{C_0}{\log n}. \]
\end{cor}
This is an improvement and generalization upon~\eqref{stolthm}.

\begin{proof}
By solving for $\varepsilon$ in $n < B \cdot C^{1/\varepsilon}$, one
    easily sees that $\varepsilon < \frac{\log C }{\log n -\log B}$.
Without loss of generality we may assume that  $B > 1$, hence we can take $C_0 = \log C $.
\end{proof}

One might expect that the ratio $s_q(p(n))/s_q(n)$ is small only rarely, with most of its time being spent near $h = \deg p$.
It turns out that this ratio is small somewhat more often than expected.

\begin{theo} \label{mtheo2.5}
For any $\varepsilon > 0$ there exists an explicitly computable $\alpha > 0$, dependent only on $\varepsilon$, $p(x)$ and $q$,
such that
 \[ \# \left\{n < N : \quad \frac{s_q(p(n))}{s_q(n)} < \varepsilon
    \right\} \gg N^\alpha  \]
where the implied constant also only depends on $\varepsilon$, $p(x)$ and $q$.
\end{theo}
The proof of this result is given in Section \ref{sec:mtheo2.5}.


In Section \ref{sec:conc} we collect together questions raised in this paper
    and pose some further lines of inquiry for this research.

\section{Preliminaries and Proof of Theorem~\ref{limsupsto}}\label{sec:prelim}

First we prove some preliminary results about $s_q$ which we need in the proofs.
Recall (cf.~\cite{Li97}) that terms are said to be \textit{noninterfering} if we can use the following splitting formul\ae:
\begin{prop}\label{propsplit}
For $1\leq b<q^k$ and $a,k\geq 1$,
  \begin{align}
    s_q(aq^k+b)&=s_q(a)+s_q(b),\label{splitpos}\\
    s_q(aq^k-b)&=s_q(a-1)+(q-1)k-s_q(b-1).\label{splitneg}
  \end{align}
\end{prop}
\begin{proof} Relation~(\ref{splitpos}) is a consequence of the (strong) $q$-additivity of $s_q$. For~(\ref{splitneg}) we write
$b-1=\sum^{k-1}_{i=0}b_iq^i$ with $0\leq b_i\leq q-1$. Then
\begin{align*}
s_q(aq^k-b) &=s_q((a-1)q^k+q^k-b)=s_q(a-1)+s_q(q^k-b)\\
&=s_q(a-1)+s_q\left(\sum_{i=0}^{k-1} (q-1-b_i)q^i\right)\\
&=s_q(a-1)+\sum_{i=0}^{k-1} (q-1-b_i)
\end{align*}
implying \eqref{splitneg}.
\end{proof}

\begin{prop}\label{propsub}
The function $s_q$ is subadditive and submultiplicative, i.e., for all $a,b\in \mathbb{N}$ we have
\begin{align}
  s_q(a+b) &\leq s_q(a)+s_q(b),\label{subadd}\\
  s_q(ab) &\leq s_q(a) s_q(b).\label{submult}
\end{align}
\end{prop}
\begin{proof}
  The proof follows on the lines of~\cite[Section~2]{Ri08}.
As for~(\ref{subadd}), an even stronger result is true,
namely that
    $s_q(a+b) = s_q(a) + s_q(b) - (q-1) \cdot r$
where $r$ is the number of ``carry'' operations needed when adding $a$
    and $b$.
  Writing $b=\sum_{i=0}^{k-1} b_i q^i$ we also have
  \begin{align*}
    s_q(ab)&=s_q\left(a\sum_{i=0}^{k-1} b_iq^i\right)\leq \sum_{i=0}^{k-1} s_q(ab_i)\\
    &=\sum_{i=0}^{k-1} s_q(\underbrace{a+\dots+a}_{\mbox{$b_i$ times}})\leq s(a)\sum_{i=0}^{k-1} b_i,
  \end{align*}
  where we used twice the subadditivity of $s_q$ and we get~(\ref{submult}).
\end{proof}

\begin{proof}[Proof of Theorem \ref{limsupsto}]
This is an almost direct generalization of Stolarsky's proof (see~\cite[Section~2]{St78}) and Propositions~\ref{propsplit}
and~\ref{propsub}. First, suppose that $p(n)$ has only nonnegative coefficients. Then using
Proposition~\ref{propsub} we see that $s_q(p(n))\leq p(s_q(n))$. Therefore
\begin{align}
  \frac{s_q(p(n))}{s_q(n)}
  &\leq\frac{\min\{(q-1) \left(\log_q p(n)+1\right), p(s_q(n))\}}{s_q(n)}\nonumber\\
  &\leq c_1\cdot \frac{\min\{\log_q n, s_q(n)^h\}}{s_q(n)} \label{stolcases}
\end{align}
where $c_1$ only depends on $p(x)$ and $q$.
If $\log_q n \leq s_q(n)^h$ then we have  $(\log_q n)^{1/h} \leq s_q(n)$.
From this and~(\ref{stolcases}), we get that
$$  \frac{s_q(p(n))}{s_q(n)}
  \leq c_1\cdot \frac{\log_q n}{(\log_q n)^{1/h}}
  = c_1 (\log_q n)^{1-1/h}.
$$
Alternately, if $\log_q n > s_q(n)^h$ then we have $(\log_q
n)^{1/h}
> s_q(n)$ and
$$  \frac{s_q(p(n))}{s_q(n)}
  \leq c_1\cdot s_q(n)^{h-1}
  \leq c_1 (\log_q n)^{1-1/h}.
$$
For the lower bound, set
\begin{equation}\label{kdef}
  k=\lfloor \log_q (\lambda (h+1)!)\rfloor +1,
\end{equation}
where $\lambda = \max \{a_i: 0\leq i\leq h\}$. By Stolarsky's use of the Bose-Chowla Theorem,
there are infinitely many integers $M\geq 3(k+1)$ such that there are integers $y_1, y_2, \dots, y_N$
with $N:=\lfloor (M+1)/(k+1)\rfloor-1,$ with the following three properties:

\begin{enumerate}
  \item[(i)] $1\leq y_1<y_2<\dots<y_N\leq M^h$,
  \item[(ii)] $y_i\equiv 0 \bmod (k+1)$,
  \item[(iii)]  all sums $y_{j_1}+\dots +y_{j_h}$ are distinct (\textit{distinct sum property});
    here $j_1, j_2, \ldots, j_{h} \in \{1, 2, \ldots, N\}$ with possible
    repetition.
\end{enumerate}
Note that (iii) implies the distinct sum property for all $y_{j_1}+\dots +y_{j_i}$ with $1\leq i\leq h$.
Now set
$$n=\sum_{i=1}^N q^{y_i},$$
such that
\begin{equation}\label{pnq}
  p(n)=\sum_{i=0}^h a_i n^i = \sum_{i=0}^h \sum {}^{'} a_i \alpha(i; h_1, \dots, h_N) q^{y_1h_1+\dots +y_N h_N}
\end{equation}
where the summation $\sum {}^{'}$ is over all vectors $(h_1,\dots,h_N)$  satisfying $h_1+\dots+h_N=i$, and
$\alpha(i; h_1, \dots, h_N)$ denote the multinomial coefficients $i!/(h_1!\dots h_N!)$ bounded by $i!$.
Consider~(\ref{pnq}) as a polynomial in $q$. By the distinct sum property (iii) we have for all $0\leq i\leq h$ that
$$\#\{y_1h_1+\dots+y_N h_N: \; h_1+\dots +h_N=i\}= \binom{N+i-1}{N-1}.$$
Thus the coefficients of $q^{y_1h_1+\dots +y_N h_N}= q^{R}$ with $h_1+\dots+h_N=h$ in~(\ref{pnq}) are nonzero and bounded by
\begin{equation}\label{abound}
  a_h h!+a_{h-1} (h-1)!+\dots +a_0 \leq \lambda (h+1) h!< q^k.
\end{equation}
By~(\ref{abound}) and (ii), the sums $y_1h_1+\dots +y_N h_N\equiv 0$ mod $(k+1)$ and hence the powers $q^R$ are
noninterfering and we get
$$\frac{s_q(p(n))}{s_q(n)}\geq \binom{N+h-1}{N-1}\cdot \frac{1}{N}\geq \frac{N^{h-1}}{h!}.$$
By construction,
$$\log_q n\leq y_N+1 \leq 2^{h+1} N^h (k+1)^h.$$
The claim now follows by observing that $k$ is largest for $q=2$.

Secondly suppose that $p(n)$ has at least one negative coefficient. Then the first claim follows by observing
that $s_q(p(n))\leq \lfloor \log_q p(n)\rfloor +1$ for sufficiently large $n$. For the lower bound, denote by
$a_j$ the negative coefficient with smallest index $j$, i.e., $a_j<0$ and $a_{j-l}\geq 0$ for $1\leq l\leq j $.
Then for all sufficiently large $k$ we have
\begin{align*}
  s_q(p(q^k))&=s_q(a_h q^{hk}+\dots+a_{j+1} q^{(j+1)k}+a_j q^{jk}+a_{j-1} q^{(j-1)k}+\dots +a_0)\\
  &=s_q(a_h q^{(h-j)k}+\dots+a_{j+1} q^{k}+a_j )+\sum_{l=0}^{j-1} s_q(a_l)\\
  &\geq k(q-1)-s(-a_j-1)\\
  &> k(q-1-\varepsilon).
\end{align*}
Here we have used Proposition~\ref{propsplit}.
As $s_q(q^k) = 1$ and $\log_q(q^k) = k$, the result follows.
This completes the proof of Theorem~\ref{limsupsto}.
\end{proof}

\section{Proof of Theorem~\ref{mtheo}}\label{sec:mtheo}

The proof of Theorem~\ref{mtheo} will use a construction of a sequence with noninterfering terms. First
assume that $p(x)=x^h$, $h\geq 2$ and define the polynomial
$$t_m(x)=m x^4+mx^3-x^2+mx+m$$
where $m\in \Z$ with $m\geq 3$.
By consecutively employing~(\ref{splitpos}) and~(\ref{splitneg}) we see that for all $k$ with $q^k>m$,
\begin{equation}\label{crx}
  s_q(t_m(q^k))=(q-1)k+s_q(m-1)+3s_q (m).
\end{equation}
The appearance of $k$ in~(\ref{crx}) is crucial. The next lemma lies at the heart of the proofs. We
will use it to see that $s_q(t_m(q^k)^h)$, $h\geq 2$, is independent of $k$ whenever $k$ is sufficiently large.
Furthermore, we will exploit the fact that the coefficients of $[x^i]$ in $t_m(x)^h$ are polynomials in $m$ with
alternating signs.
\begin{lem}\label{lemuseful}
  For fixed $h\geq 2$ and $m\geq 3$, we have
  $$ t_m(x)^h =\sum_{i=0}^{4h} c_{i,h}(m) \; x^i$$
  satisfying
  \begin{equation}\label{cestim}
    0<c_{i,h}(m) \leq (2mh)^h \qquad i=0,1,\dots, 4h.
  \end{equation}
In fact, we have
\begin{align}\label{coc1}
c_{0,h}(m)=c_{4h,h}(m)=m^h, \ \ c_{1,h}(m)=c_{4h-1,h}(m)=hm^h.
\end{align}
\end{lem}

\begin{proof}A direct calculation shows that $t_m(x)^2$ and $t_m(x)^3$ have
property \eqref{cestim} provided $m\geq 3$. Set $h=2h_1+3h_2$ with
$\max(h_1,h_2)\geq 1$. Then
$$t_m(x)^h = \underbrace{t_m(x)^2\dots t_m(x)^2}_{\mbox{$h_1$ times}} \cdot \underbrace{t_m(x)^3 \dots t_m(x)^3}_{\mbox{$h_2$ times}}.$$
Since products of polynomials with all positive coefficients have all positive coefficients too, we get $c_{i,h}(m)>0$ for all
$i=0,1,\dots, 4h$. On the other hand, the coefficients of $t_m(x)^h$ are
clearly bounded by the corresponding coefficients of the polynomial
\begin{align*}
m^h(1+x+x^2+x^3+x^4)^h=m^h\sum_{0\leq l\leq k\leq j\leq i\leq h}\binom{h}{i}
\binom{i}{j}\binom{j}{k}\binom{k}{l}x^{i+j+k+l}.
\end{align*}
 Therefore, for all $i$ with $0\leq i\leq 4h$, we have
\begin{align}\label{eq:c_i}
c_{i,h}(m)&\leq m^h \sum_{0\leq l\leq k\leq j\leq i\leq h} \frac{h!}{(h-i)!(i-j)!(j-k)!(k-l)!}\\
&\leq m^hh! \;{\rm exp}(h-i+i-j+j-k+k-l)\nonumber\\
&\leq m^hh!e^h\leq (2mh)^h.\nonumber
\end{align}
\end{proof}

\begin{proof}[Proof of Theorem \ref{mtheo}]
Now let $k$ be such that $q^k>(2mh)^h$. By~(\ref{cestim}) and~(\ref{splitpos}) we then have
$$s_q(t_m(q^k)^h)=s_q(c_{0,h}(m))+s_q(c_{1,h}(m))+\dots+s_q(c_{4h,h}(m))$$
where $s_q(c_{i,h}(m))$ is bounded by a function which only depends on $q$, $m$ and $h$.
Together with~(\ref{crx}) and letting $k\to \infty$ we thus conclude for fixed $m\geq 3$,
$$\lim_{k\to\infty} s_q(t_m(q^k)^h)/s_q(t_m(q^k))=0,$$
as wanted.

Finally we consider the case with a general polynomial instead of $x^h$. Write
\begin{equation}\label{gen}
p(t_m(x))=a_h t_m(x)^h+a_{h-1} t_m(x)^{h-1}+\dots+a_1 t_m(x) +a_0
\end{equation}
where $a_h>0$ and $h\geq 2$. First  suppose that all the coefficients
are nonnegative. Lemma \ref{lemuseful} shows that for $i$ with
$2\leq i\leq h$ all the coefficients of $t_m(x)^i$ are positive. Also,
the coefficient
$[x^2]$ in $p(t_m(x))$ is nonnegative if we choose $m\geq 3$ sufficiently large. In fact, a sufficient
condition is $a_h \left( \binom{h}{2} m^h-h m^{h-1}\right)\geq a_1$ which is true whenever
\begin{equation}\label{mestim}
  m\geq \left(\frac{2a_1}{h(3h-5) a_h}\right)^{1/(h-1)}.
\end{equation}
If the polynomial $p(x)$ has negative coefficients then there is a positive
integer $b$ such that the polynomial $p(x+b)$ has all positive coefficients.
A good choice for $b$ is
\begin{equation}\label{bestim}
b=\left\lceil 1+\frac{\lambda}{a_h}\right\rceil=
1+\left\lceil \frac{\lambda}{a_h}\right\rceil
,\qquad \lambda=\max \{|a_i|: 0\leq i\leq h\}.
\end{equation}
This is easy to see since both $p(x+b)-(a_h(x+b)^h-\lambda \sum^{h-1}_{i=0}
(x+b)^i)$ and
\begin{align*}
a_h(x+b)^h-&\lambda \sum^{h-1}_{i=0}(x+b)^i\\ &=
\left(a_h-\frac{\lambda}{x+b-1}\right)(x+b)^h+\frac{\lambda}{x+b-1}\\
&=\frac{1}{x+b-1}\left((a_hx+(b-1)a_h-\lambda)(x+b)^h+\lambda\right)
\end{align*}
have nonnegative coefficients when $b\geq 1+\frac{\lambda}{a_h}$. Thus
if $q^k>m+b$ then
$s_q(t_m(q^k)+b)=(q-1)k+s_q(m-1)+2s_q(m)+s_q(m+b)$ and one similarly obtains
for fixed $m$,
$$\lim_{k\to\infty} s_q(p(t_m(q^k)+b))/s_q(t_m(q^k)+b)=0.$$
This completes the proof of Theorem~\ref{mtheo}.
\end{proof}

\section{Proof of Theorem~\ref{mtheo2}}\label{sec:mtheo2}

The construction of an extremal sequence in the proof of Theorem~\ref{mtheo} gives a rough
bound on the minimal $n$ such that $s_q(n^h)<s_q(n).$ We first illustrate the method in the case $q=2$, $h=3$.

Set $m=3$. Then for all $k$ with $2^k>\max\limits_{0\leq i\leq 4h}{c_{i,h}(m)}=225$ we have
\begin{align*}
  s_2(t_3(2^k)) &= k+1+6=k+7,\\
  s_2(t_3(2^k)^3) &= 2\cdot(4+3+4+4+4+4)+4=50.
\end{align*}
Therefore, by setting $k=44$, we get
$$\min \{n: s_2(n^3)<s_2(n)\} < 2^{178}.$$
It is possible to show that the minimal such $n$ to be
    $n=407182835067\approx 2^{39}$.

\begin{proof}[Proof of Theorem \ref{mtheo2}]
Consider the general polynomial
$$p(x)=a_hx^h+a_{h-1}x^{h-1}+\dots+a_0 \in \Z[x]$$
with $a_h>0$, $h\geq 2$.
Let $\lambda = \max |a_i|$.
Pick $b$ such that $p(x+b)$ has only nonnegative coefficients, as in \eqref{bestim}.
Pick $m \geq 3$ such that $p(t_m(x)+b)$ has only nonnegative coefficients,
    as in \eqref{mestim}.
Our task is to bound the coefficients of of $p(t_m(x)+b)\in \Z[x]$.

To begin with, we estimate the coefficient of $x^i, 0\leq i\leq h$ of
$p(x+b)$,
\begin{equation}\label{pyestim}
\sum_{j=i}^h a_j b^{j-i} \binom{j}{i} \leq \sum_{j=i}^h \left\vert a_j b^{j-i} \binom{j}{i} \right\vert \leq \lambda (2b)^h.
\end{equation}
Combining \eqref{pyestim} with \eqref{coc1}, we find that the
constant term of $p(t_m(x)+b)$ is bounded by
\begin{equation*}
\lambda (2b)^h \sum^h_{i=0} m^i=\lambda (2b)^h\frac{m^{h+1}-1}{m-1}
 \leq \lambda h(4mbh)^h
\end{equation*}
since $m\geq 3$ and $h\geq 2$. Again from \eqref{pyestim} and
\eqref{cestim}, we find that the other coefficients of
$p(t_m(x)+b)$ are bounded by
\begin{equation}\label{finalestim}
  \lambda (2b)^h \sum_{i=1}^h (2mi)^i \leq \lambda h(4 mbh)^h.
\end{equation}
Therefore the coefficients of $p(t_m(x)+b)$ are bounded by
$\lambda h(4 mbh)^h$. Hence for $q^k>m+b$, we have
\begin{equation}\label{sqhestim}
s_q(p(t_m(q^k))+b)\leq (q-1)(4h+1)
\left(\frac{\log(\lambda h(4mbh)^h)}{\log q}+1\right).
\end{equation}

On the other hand, we clearly have $s_q(t_m(q^k)+b)> (q-1)k$ for $q^k>m+b$.
Let
\begin{align*}
k=\left\lfloor \frac{4h+1}{\varepsilon}
\left(\frac{\log(\lambda h(4mbh)^h)}{\log q}+1\right)\right\rfloor +1.
\end{align*}
Then for $0<\varepsilon<h(4h+1)$ we have $q^k>m+b$ and hence
$$\frac{s_q(p(t_m(q^k)+b))}{s_q(t_m(q^k)+b)} <\varepsilon.$$
Therefore,
\begin{align*}
  \min \left\{n:\; \frac{s_q(p(n))}{s_q(n)}<\varepsilon\right\}
  &\leq t_m(q^k)+b\\
  &< m(q^{4k}+q^{3k}+q^k+1)\\
  &<2m q^{4k}\\
  &\leq 2mq^4 \left(q\lambda h(4mbh)^h\right)^{(16h+4)/\varepsilon}.
\end{align*}
Setting
    $B := 2 m q^4$ and
    $C := \left(q\lambda h(4mbh)^h\right)^{16h+4}$, it gives the desired
    result.
\end{proof}

\section{Proof of Theorem~\ref{mtheo2.5}}\label{sec:mtheo2.5}

We start our analysis with the simple case of $p(n) = n^h$. Let $t_m(x) = m x^4 + m x^3 - x^2 + m x + m$
as in Section~\ref{sec:mtheo}.
Letting $n = n_{k,m} = t_m(q^k)$ we see from equation~(\ref{crx}) that, for $m<q^k$,
     \[ s_q(n) = (q-1) k + s_q(m-1) + 3 s_q(m) \geq (q-1) k.  \]
If $m$ has $i$ $q$-ary digits then $n$ will have
    $4 k + i$ $q$-ary digits. We see that $t_m(q^k)^h$ is of length at most $h (4 k + i)$.

Let $t_m(q^k)^h = \sum_{j=0}^{4h} c_{j} q^{kj}$.
These $c_{j}$ are dependent upon $m$ and $h$, but are independent of $k$ for $k$ sufficiently large.
We see from equation (\ref{eq:c_i}) that
    $c_j \leq (m h \cdot 2)^h$ and hence has
    at most $h i + h \log_q h + h$ $q$-ary digits.
As there are $(4 h + 1)$ coefficients $c_j$ and
$s_q(c_j)\leq (q-1)(hi + h\log_q h + h)$, we get
    $$s_q(n^h) \leq (q-1)  (4h+1) \left(h i + h \log_q h  + h  \right).$$

Combining these together we have
\begin{align*}
     \frac{s_q(n^h)}{s_q(n)} &\leq
       \frac{(q-1) (4h+1) \left(hi + h\log_q h + h  \right)}{(q-1) k}\\
       & =
       \frac{(4h+1) \left(h i + h\log_q h + h  \right)}{k}.
\end{align*}
Without loss of generality suppose that $0<\varepsilon<h(4h+1)$.
Let $k_0$ be large enough so that $k_0>i$ and
 \[\frac{(4h+1) \left(h i + h \log_q h + h\right)}{k_0}   < \varepsilon. \]
For $i$ sufficiently large, we can take $k_0 =
    \left \lfloor \frac{(4h+1)(hi+i)}{\varepsilon}\right \rfloor $.
Then this says that for every sufficiently large $m$ having $i$ $q-$ary digits,
there is an integer $n$ having $4 k_0 + i$ $q-$ary digits such that
    \[ \frac{s_q(n^h)}{s_q(n)} < \varepsilon. \]
Moreover, by construction, each distinct $m$ will give rise to a distinct
    $n$. Letting \[ \alpha = \frac{i}{4 k_0 + i} \geq
    \frac{i}{4(4h+1)  (h+1)i/\varepsilon + i} =
    \frac{\varepsilon}{4 (4h+1) (h+1) + \varepsilon} \]
we get as $N\to \infty$ that
 \[ \# \left\{n < N : \quad \frac{s_q(n^h)}{s_q(n)} < \varepsilon
    \right\} \gg N^\alpha. \]

Now to extend this for general $p(x)$, we proceed as we did in the proof
    of Theorem \ref{mtheo}.
First consider the case where $p(x)$ has only nonnegative coefficients.
There is a lower bound on $m$ such that $p(t_m(x))$ will have only
    nonnegative coefficients and we proceed as before, after which the
    result follows as before.
Second, if $p(x)$ has at least one negative coefficient, then consider instead the polynomial
    $p(x+b)$ for sufficiently large $b$, which will have only nonnegative
    coefficients, and the result follows.

\section{Conclusions and further work}
\label{sec:conc}

All results in this paper have explicitly computable constants for
existence or density results. Many times these constants are far
from the observed experimental values, and it is quite likely that
many of them may be strengthened. Examples include
Theorems~\ref{mtheo2} and~\ref{mtheo2.5}.

Some obvious generalizations of this problem are in looking at the ratios of
$\frac{s_q(p_1(n))}{s_q(p_2(n))}$, or even more generally
of $\frac{s_{q_1}(p_1(n))}{s_{q_2}(p_2(n))}$ with respect to two different
bases $q_1, q_2$. Alternately, instead of looking at polynomials $p(x) \in \Z[x]$,
we could look at quasi-polynomials $\lfloor p(n) \rfloor$ with $p(x) \in \R[x]$.

As another direction, we could consider expansions in other numeration
systems, e.g. the Zeckendorf expansion (or expansions with respect to linear
recurrences) or the balanced based $q$ representation. In the latter
case, for example, $11 = 1 \cdot 3^2  + 1 \cdot 3^1 - 1 \cdot 3^0$, and  $s_3'(11) = 1 + 1 - 1 = 1$,
being the sum-of-digits function in this representation. This value will quite often be $0$, but its
extremal distribution could still have some interesting properties.

\section*{Acknowledgements.}

The authors thank J. Shallit for his remarks on a previous version of this paper.


\begin{thebibliography}{9}

\bibitem{BK95}
N. L. Bassily, I. K\'atai, Distribution of the values of $q$-additive functions on polynomial sequences,
\textit{Acta Math. Hung.} \textbf{68} (1995), 353--361.

\bibitem{BC62}
R. C. Bose, S. Chowla, Theorems in the additive theory of numbers,
\textit{Comm. Math. Helv.} \textbf{37} (1962/63), 141--147.

\bibitem{DT06}
C. Dartyge, G. Tenenbaum, Congruences de sommes de chiffres de valeurs polynomiales,
\textit{Bull. London Math. Soc.} \textbf{38} (2006), no. 1, 61--69.

\bibitem{De75}
H. Delange, Sur la fonction sommatoire de la fonction ``somme des chiffres'', \textit{Enseign. Math.} \textbf{21} (1975), 31--47.

\bibitem{DR05}
M. Drmota, J. Rivat, The sum-of-digits function of squares, \textit{J. London Math. Soc. (2)} \textbf{72} (2005), no. 2, 273--292.

\bibitem{HR83}
H. Halberstam, K. F. Roth, \textit{Sequences}, Second edition. Springer-Verlag, New York-Berlin, 1983.

\bibitem{HLS}
K. G. Hare, S. Laishram, T. Stoll, The sum of digits of $n$ and $n^2$,
    submitted, arxiv.org

\bibitem{Li97}
B. Lindstr\"om, On the binary digits of a power, \textit{J. Number Theory} \textbf{65} (1997), 321--324.

\bibitem{MR09-1}
C. Mauduit,  J. Rivat, Sur un probl\'eme de Gelfond: la somme des chiffres des nombres premiers, \textit{Annals of Mathematics}, to appear.

\bibitem{MR09-2}
C. Mauduit,  J. Rivat, La somme des chiffres des carr\'es, \textit{Acta Mathematica} \textbf{203} (2009), 107--148.

\bibitem{Me05}
G. Melfi, On simultaneous binary expansions of $n$ and $n^2$, \textit{J. Number Theory} \textbf{111} (2005), no. 2, 248--256.

\bibitem{Pe02}
M. Peter, The summatory function of the sum-of-digits function on polynomial sequences, \textit{Acta Arith.} \textbf{104} (2002), no. 1, 85--96.

\bibitem{Ri08}
T. Rivoal, On the bits counting function of real numbers, \textit{J. Aust. Math. Soc.} \textbf{85} (2008), no. 1, 95--111.

\bibitem{St78}
K. B. Stolarsky, The binary digits of a power, \textit{Proc. Amer. Math. Soc.} \textbf{71} (1978), 1--5.

\end{thebibliography}
\end{document}